\documentclass{amsart}
\usepackage{amsthm,hyperref}
\usepackage{amsmath,amssymb,MnSymbol}

\numberwithin{equation}{section}
\newtheorem{thm}{Theorem}[section]
\newtheorem{con}[thm]{Conjecture}
\newtheorem{cor}[thm]{Corollary}
\newtheorem{lem}[thm]{Lemma}
\newtheorem{prop}[thm]{Proposition}

\theoremstyle{definition}
\newtheorem{defn}[thm]{Definition}

\newtheorem{rem}[thm]{Remark}
\newtheorem{exa}[thm]{Example}

\DeclareMathOperator{\Mon}{Mon}
\DeclareMathOperator{\lex}{lex}

\DeclareMathOperator{\supp}{supp}

\begin{document}
\title{\textbf{On The Eisenbud-Green-Harris Conjecture}}
\author[A.Abedelfatah]{Abed Abedelfatah}
\address{Department of Mathematics, University of Haifa, Mount Carmel, Haifa 31905, Israel}
\email{abed@math.haifa.ac.il}
\begin{abstract}
It has been conjectured by Eisenbud, Green and Harris that if $I$ is a homogeneous ideal in $k[x_1,\dots,x_n]$ containing a regular sequence $f_1,\dots,f_n$ of degrees $\deg(f_i)=a_i$, where $2\leq a_1\leq \cdots \leq a_n$, then there is a homogeneous ideal $J$ containing $x_1^{a_1},\dots,x_n^{a_n}$ with the same Hilbert function. In this paper we prove the Eisenbud-Green-Harris conjecture when $f_i$ splits into linear factors for all $i$.
\end{abstract}
\maketitle

\section{Introduction}
Let $S=k[x_1,\dots,x_n]$ be a polynomial ring over a field $k$. The ring $S=\bigoplus_{d\geq0}S_d$ is graded by $\deg(x_i)=1$ for all $i$. In 1927, F.Macaulay proved that if $I=\bigoplus_{d\geq0}I_d$ is a graded ideal in $S$, then there exists a lex ideal $L$ such that $L$ has the same Hilbert function as $I$ \cite{macaulay}; i.e., every Hilbert function in $S$ is attained by a lex ideal. Let $M$ be a monomial ideal in $S$. It is natural to ask if we have the same result in $S/M$. In \cite{clements}, Clements and Lindstr{\"o}m proved that every Hilbert function in $S/\langle x_1^{a_1},\dots,x_n^{a_n}\rangle$ is attained by a lex ideal, where $2\leq a_1\leq \dots \leq a_n$. In the case $a_1=\cdots=a_n=2$, the result was obtained earlier by Katona \cite{katona} and Kruskal \cite{Kruskal}. Another generalizations of Macaulay's theorem can be found in \cite{shakin}, \cite{mermin2} and \cite{abed}.

Let $f_1,\dots,f_n$ be a regular sequence in $S$ such that $2\leq a_1=\deg(f_1)\leq \cdots\leq a_n=\deg(f_n)$. A well known result says that $\langle f_1,\dots,f_n\rangle$ has the same Hilbert function as $\langle x_1^{a_1},\dots,x_n^{a_n}\rangle$ (see Exercise 6.2. of \cite{monomial}). It is natural to ask what happens if $I\subseteq S$ is a homogeneous ideal containing a regular sequence in fixed degrees. This question bring us to the Eisenbud-Green-Harris Conjecture, denoted by EGH.

\begin{con}[EGH \cite{eisenbud}]\label{1}{\ \\}
If $I$ is a homogeneous ideal in $S$ containing a regular sequence $f_1,\dots,f_n$ of degrees $\deg(f_i)=a_i$, where $2\leq a_1\leq\cdots\leq a_n$, then $I$ has the same Hilbert function as an ideal containing $x_1^{a_1},\dots,x_n^{a_n}$.
\end{con}

The original conjecture (see Conjecture \ref{18}) is equivalent to Conjecture \ref{1} in the case $a_i=2$ for all $i$ (see Proposition \ref{19}). The EGH Conjecture is known to be true in few cases. The conjecture has been proven in the case $n=2$ \cite{richert}. Caviglia and Maclagan \cite{caviglia} have proven that the EGH Conjecture is true if $a_j>\sum_{i=1}^{j-1}(a_i-1)$ for all $j>1$. Richert \cite{richert} says that the EGH Conjecture in degree 2 ($a_i=2$ for all $i$) holds for $n\leq5$, but this result was not published. Herzog and Popescu \cite{herzog} proved that if $k$ is a field of characteristic zero and $I$ is minimally generated by generic quadratic forms, then the EGH Conjecture in degree 2 holds. Cooper \cite{cooper1,cooper2} has done some work in a geometric direction. She studies the EGH Conjecture for some cases with $n=3$.

Let $f_1,\dots,f_n$ be a regular sequence in $S$ such that $f_i$ splits into linear factors for all $i$. For all $1\leq i\leq n$, let $p_i\in S_1$ such that $p_i|f_i$. Since $p_1,\dots,p_n$ must be a $k$-linear independent, it follows that the $k$-algebra map $\alpha:S\rightarrow S$ defined by $\alpha(x_i)=p_i$ for all $1\leq i\leq n$, is a graded isomorphism. So the Hilbert function is preserved under this map and we may assume that $p_i=x_i$ for all $i$.

In Section 2, we give background information to the EGH Conjecture. In section 3, we study the dimension growth of some ideals containing a regular sequence $x_1l_1,\dots,x_nl_n$, where $l_i\in S_1$ for all $i$. In section 4, we prove the EGH Conjecture when $f_i$ splits into linear factors for all $i$. This answers a question of Chen, who asked if the EGH Conjecture holds when $f_i=x_il_i$, where $l_i\in S_1$ for all $1\leq i\leq n$ (see Example 3.8 of \cite{chen}).

\section{Background}
A proper ideal $I$ in $S$ is called \emph{graded} or \emph{homogeneous} if it has a system of homogeneous generators. Let $R=S/I$, where $I$ is a homogeneous ideal. The \emph{Hilbert function} of $I$ is the sequence $H(R)=\{H(R,t)\}_{t\geq0}$, where $$H(R,t):=\dim_{k}R_t=\dim_{k}S_t/I_t.$$
For simplicity, sometimes we denote the dimension of a $k$-vector space $V$ by $|V|$ instead of $\dim_{k}V$. For a $k$-vector space $V\subseteq S_d$, where $d\geq 0$, we denote by $S_1V$ the $k$-vector space spanned by $\{x_iv:~1\leq i\leq n~\wedge~v\in V\}$. Throughout this paper $\textbf{A}=(a_1,\dots,a_n)\in \mathbb{Z}^n$, where $2\leq a_1\leq\cdots\leq a_n$. For a subset $A$ of $S$, we denote by $\Mon(A)$ the set of all monomials in $A$ and let $A_u=\{j:~x_j|u\}$, where $u\in \Mon(S)$. The \emph{support} of the polynomial $f=\sum_{u\in\Mon(S)}a_uu$, where $a_u\in k$, is the set $$\supp(f)=\{u\in \Mon(S):~a_u\neq0\}.$$
A monomial $w\in S$ is called \emph{square-free} if $x_i^2\nmid w$, for all $1\leq i\leq n$. We define the \emph{lex order} on $\Mon(S)$ by setting $\textbf{x}^b=x_1^{b_1}\cdots x_n^{b_n}<_{\lex}x_1^{c_1}\cdots x_n^{c_n}=\textbf{x}^c$ if either $\deg(\textbf{x}^b)<\deg(\textbf{x}^c)$ or $\deg(\textbf{x}^b)=\deg(\textbf{x}^c)$ and $b_i<c_i$ for the first index $i$ such that $b_i\neq c_i$. We recall the definitions of lex ideal and lex-plus-powers ideal.

\begin{defn}
\begin{itemize}
  \item A graded ideal is called \emph{monomial} if it has a system of monomial generators.
  \item A monomial ideal $I\subseteq S$ is called \emph{lex}, if whenever $I\ni z<_{\lex}w$, where $w,z$ are monomials of the same degree, then $w\in I$.
  \item A monomial ideal $I$ is \emph{$\textbf{A}$-lex-plus-powers} if there exists a lex ideal $L$ such that $I=\langle x_1^{a_1},\dots,x_n^{a_n}\rangle+L$.
\end{itemize}
\end{defn}

\begin{exa}
the ideal $I=\langle x_1^2,x_2^2,x_1x_2x_3,x_3^3\rangle$ is a $(2,2,3)$-lex-plus-powers ideal in $k[x_1,x_2,x_3]$, because $I=\langle x_1^2,x_2^2,x_3^3\rangle+\langle x_1^3,x_1^2x_2,x_1^2x_3,x_1x_2^2,x_1x_2x_3\rangle$ and\\ $\langle x_1^3,x_1^2x_2,x_1^2x_3,x_1x_2^2,x_1x_2x_3\rangle$ is a lex ideal in $k[x_1,x_2,x_3]$.
\end{exa}

By Clements-Lindstr{\"o}m's theorem, we obtain that for any graded ideal containing $\langle x_1^{a_1},\dots,x_n^{a_n}\rangle$ there is a $(a_1,\dots,a_n)$-lex-plus-powers ideal with the same Hilbert function.

Let $p\geq0$ and ${s_q \choose q}+{s_{q-1}\choose q-1}+\cdots+{s_1\choose 1}$ be the unique Macaulay expansion of $p$ with respect to $q>0$. Set $0^{(q)}=0$ and $p^{(q)}={s_q \choose q+1}+{s_{q-1}\choose q}+\cdots+{s_1\choose 2}$. In \cite{eisenbud}, Eisenbud, Green and Harris made the following conjecture.

\begin{con}\label{18}
If $I\subset S$ is a graded ideal such that $I_2$ contains a regular sequence of maximal length and $d>0$, then $H(S/I,d+1)\leq H(S/I,d)^{(d)}$.
\end{con}

Conjecture \ref{18} is true if the ideal contains the squares of the variables. This follows from the Kruskal-Katona theorem (see \cite{gotzmann}). In the following proposition, we prove the equivalence of Conjecture \ref{18} and the EGH Conjecture in degree 2. First, we need the following definition.

\begin{defn}
Let $M$ be a monomial ideal in $S$ and $d\geq0$. A monomial vector space $L_d$ in $(S/M)_d$ is called \emph{lexsegment} if it is generated by the $t$ biggest monomials (with respect to the lex order) in $(S/M)_d=S_d/M_d$, for some $t\geq0$.
\end{defn}

For example, if $L$ is a lex ideal in $S$, then $L_j$ is lexsegment for all $j\geq0$. If $L_d$ is a lexsegment space in $(S/M)_d$, where $M$ is a monomial ideal in $S$, then $S_1L_d$ is lexsegment in $(S/M)_{d+1}$ (see Proposition 2.5 of \cite{mermin2}).

\begin{prop}\label{19}
Let $f_1,\dots,f_n$ be a regular sequence of degrees 2 in $S$. The following are equivalent:
\begin{itemize}
  \item [(a)] If $I$ is a graded ideal in $S$ containing $f_1,\dots,f_n$, then there is a graded ideal $J$ in $S$ containing $x_1^2,\dots,x_n^2$ such that $H(S/I)=H(S/J)$.
  \item [(b)] If $I$ is a graded ideal in $S$ containing $f_1,\dots,f_n$, then \begin{center}$H(S/I,d+1)\leq H(S/I,d)^{(d)}$ for all $d>0$.\end{center}
\end{itemize}
\end{prop}

\begin{proof}
First, we prove that (a) implies (b). Let $I$ be a graded ideal in $S$ containing $f_1,\dots.f_n$. By (a), it follows that there is a graded ideal $J$ in $S$ containing $x_1^2,\dots,x_n^2$ such that $H(S/I)=H(S/J)$. By Kruskal-Katona theorem it follows that $H(S/I,d+1)=H(S/J,d+1)\leq H(S/J,d)^{(d)}=H(S/I,d)^{(d)}$ for all $d>0$.

Now, we prove that (b) implies (a). Let $I$ be a graded ideal in $S$ containing $f_1,\dots,f_n$. Set $M=\langle x_1^2,\dots,x_n^2\rangle$ and $P=\langle f_1,\dots,f_n\rangle$. For every $d\geq0$, let $L_d$ be the $k$-vector space spanned by the first square-free monomials (in lex order) of $S_d$ such that $|L_d\oplus M_d|=|I_d|$. Let $K=\oplus_{j\geq0}K_j=\oplus_{j\geq0}(L_j+M_j)$. We need to show that $K$ is an ideal. Let $d>0$. By Proposition 6.4.3 of \cite{monomial}, we obtain that $$|S_{d+1}/M_{d+1}|-|S_1L_d/S_1L_d\cap M_{d+1}|=(|S_d/M_d|-|L_d|)^{(d)}.$$ By the hypothesis of (b), we obtain $(|S_d/M_d|-|L_d|)^{(d)}=|S_d/I_d|^{(d)}\geq |S_{d+1}/I_{d+1}|$. So $$|S_{d+1}/M_{d+1}|-|S_1L_d/S_1L_d\cap M_{d+1}|=|S_{d+1}/M_{d+1}|-|S_1L_d+M_{d+1}/M_{d+1}|\geq|S_{d+1}/I_{d+1}|.$$ This implies that $|S_1L_d+M_{d+1}|\leq|L_{d+1}+M_{d+1}|$. Since $\overline{L_{d+1}}$ and $\overline{S_1L_d}$ are lexsegments in $(S/M)_{d+1}$, it follows that $S_1L_d\subseteq L_{d+1}+M_{d+1}$. So $S_1K_d\subseteq K_{d+1}$ for all $d\geq0$, which implies that $K$ is a graded ideal in $S$. Clearly, $H(S/K)=H(S/I)$.
\end{proof}

The following lemma helps us to study the EGH Conjecture in each component of the homogeneous ideal.
\begin{lem}\label{2}
Let $I$ be a graded ideal in $S$ containing a regular sequence $f_1,\dots,f_n$ of degrees $\deg(f_i)=a_i$. The following are equivalent:
\begin{itemize}
  \item [(a)] There exists a graded ideal $J$ in $S$ containing $x_1^{a_1},\dots,x_n^{a_n}$ such that $H(S/I)=H(S/J)$.
  \item [(b)] For every $d\geq0$, there exists a graded ideal $J$ in $S$ containing $x_1^{a_1},\dots,x_n^{a_n}$ such that $H(S/I,d)=H(S/J,d)$ and  $H(S/I,d+1)\leq H(S/J,d+1)$.
\end{itemize}
\end{lem}

\begin{proof}
Clearly, (a) implies (b). We will show that (b) implies (a). For every $d\geq0$, there exists an ideal $J_d$ in $S$ containing $x_1^{a_1},\dots,x_n^{a_n}$ such that $H(S/I,d)=H(S/J_d,d)$ and  $H(S/I,d+1)\leq H(S/J_d,d+1)$. By Clements-Lindstr{\"o}m's theorem, we may assume that $J_d$ is a $\textbf{A}$-lex-plus-powers ideal for all $d$. Let $J=\oplus_{j\geq0}J_{j,j}$, where $J_{j,j}$ is the $j$-th component of $J_j$. Since $\dim(J_{d,d+1})\leq\dim(I_{d+1})=\dim(J_{d+1,d+1})$, it follows that $J_{d,d+1}\subseteq J_{d+1,d+1}$, for all $d$. So $S_1J_{d,d}\subseteq J_{d,d+1}\subseteq J_{d+1,d+1}$, for all $d$. Thus, $J$ is an ideal. Clearly, $H(S/I)=H(S/J)$.
\end{proof}

We will use the following lemma on regular sequences (see \cite[Chapter 6]{matsumura}).

\begin{lem}\label{3} Let $f_1,\dots,f_n$ be a sequence of homogeneous polynomials in $S$ with $\deg(f_i)=a_i$ and $P=\langle f_1,\dots,f_n\rangle$. Then
\begin{itemize}
  \item[(a)] If $f_1,\dots,f_n$ is a regular sequence, then $H(S/P)=H(S/\langle x_1^{a_1},\dots,x_n^{a_n}\rangle)$.
  \item[(b)] $f_1,\dots,f_n$ is a regular sequence if and only if the following condition holds:\\ if $g_1f_1+\cdots +g_nf_n=0$ for some $g_1,\dots,g_n\in S$, then $g_1,\dots,g_n\in P$.
  \item[(c)] If $f_1,\dots,f_n$ is a regular sequence and $\sigma\in S_n$ is a permutation, then\\ $f_{\sigma(1)},\dots,f_{\sigma(n)}$ is a regular sequence.
\end{itemize}
\end{lem}

\section{The dimension growth of some ideals containing a reducible regular sequence}
Let $f_1=x_1l_1,\dots,f_n=x_nl_n$ be a regular sequence in $S$, where $l_i\in S_1$ for all $i$. Set $P=\langle f_1,\dots,f_n\rangle$ and $M=\langle x_1^2,\dots,x_n^2\rangle$. Let $V_d$ be a vector space spanned by $P_d$ and square-free monomials $w_1,\dots,w_t$ in $S_d$, and $W_d$ be the vector space spanned by $M_d$ and $w_1,\dots,w_t$. In this section, we prove that $\dim(S_1V_d)=\dim(S_1W_d)$. We also compute $\dim(S_1K_d)$, where $K_d$ is the space generated by $P_d$ and the biggest (in lex order) square-free monomials $v_1,\dots,v_t$ in $S_d$.

For a matrix $A\in M_{n\times n}(k)$, we denote by $A[i_1,\dots,i_r]$ the submatrix of $A$ formed by rows $i_1,\dots,i_r$ and columns $i_1,\dots,i_r$, where $1\leq r\leq n$ and $1\leq i_1<\cdots <i_r\leq n$. We begin with the following lemma, which characterize the structure of $f_1,\dots,f_n$.

\begin{lem}\emph{(Example 3.8 of \cite{chen})}\\\label{4}
Let $f_1=x_1l_1,\dots,f_n=x_nl_n$ be a sequence of homogeneous polynomials in $S$, where $l_i=\sum_{j=1}^{n}a_{ij}x_j$ with $a_{ij}\in k$ and $A$ be the $n\times n$ matrix $(a_{ij})$. Then $f_1,\dots,f_n$ is a regular sequence if and only if $\det A[i_1,\dots,i_r]\neq0$ for all $1\leq r\leq n$ and $1\leq i_1<\cdots <i_r\leq n$.
\end{lem}

\begin{proof}
Assume that $f_1,\dots,f_n$ is regular. We prove that $\det A[i_1,\dots,i_r]\neq0$ for all $1\leq r\leq n$ and $1\leq i_1<\cdots <i_r\leq n$, by induction on $n$, starting with $n=1$. Let $n>1$. Assume that $1\leq i_1<\cdots <i_r\leq n$, where $1\leq r\leq n-1$. Let $j\notin \{i_1,\dots,i_r\}$. Note that $x_jl_j$ is regular modulo an ideal $I$ if and only if both $x_j$ and $l_j$ are regular modulo $I$. By Lemma \ref{3}, $x_j,f_1,\dots,f_{j-1},f_{j+1},\dots,f_{n}$ is a regular sequence. So $\overline{f_1},\dots,\overline{f_{j-1}},\overline{f_{j+1}},\dots,\overline{f_{n}}$ is a regular sequence in $S/\langle x_j\rangle$. By the inductive step we obtain that $\det A[i_1,\dots,i_r]\neq0$. It remains to show that $\det(A)\neq0$. From the permutability property of regular sequences of homogeneous polynomials, we obtain that $l_1,\dots,l_n$ is a regular sequence. So $l_1,\dots,l_n$ is $k$-linearly independent.

Assume now $\det A[i_1,\dots,i_r]\neq0$ for all $1\leq r\leq n$ and $1\leq i_1<\cdots <i_r\leq n$. We prove that $f_1,\dots,f_n$ is a regular sequence by induction on $n$, starting with $n=1$. Let $n>1$. By the inductive step, the sequence $\overline{f_1},\dots,\overline{f_{n-1}}$ is regular in $S/\langle x_n\rangle$. So $f_1,\dots,f_{n-1},x_n$ is a regulae sequence in $S$. It remains to show that $f_1,\dots,f_{n-1},l_n$ is a regular sequence. Since $\det(A)\neq0$, it follows that the $k$-algebra map $\alpha: S\rightarrow S$ defined by $\alpha(x_i)=l_i$, for all $i$, is an isomorphism. By the inductive step, $\alpha^{-1}(f_1),\dots,\alpha^{-1}(f_{n-1}),\alpha^{-1}(l_n)=x_n$ is a regular sequence. So $f_1,\dots,f_{n-1},l_n$ is a regular sequence, as desired.
\end{proof}

The special structure of the regular sequence in \ref{4} implies the following lemma.
\begin{lem}\label{5}
Let $f_1=x_1l_1,\dots,f_n=x_nl_n$ be a regular sequence of homogeneous polynomials in $S$, where $l_i=\sum_{j=1}^{n}a_{ij}x_j$ with $a_{ij}\in k$, and $P=\langle f_1,\dots,f_n\rangle$. If $g\notin P$ is a homogeneous polynomial in $S$, then $$g\equiv h\pmod P$$ where $\deg(h)=\deg(g)$ and $h$ is a $k$-linear combination of square-free monomials.
\end{lem}

\begin{proof}
Since $g\notin P$, we have $\deg(g)\leq n$. It is sufficient to prove the lemma when $g\notin P$ is a monomial in $\langle x_1^2,\dots,x_n^2\rangle$ of degree $\leq n$. We prove by induction on $\deg(g)$. The lemma is true when $\deg(g)=2$, since $a_{ii}\neq0$ for all $i$. Let $g$ be a monomial in $\langle x_1^2,\dots,x_n^2\rangle$ of degree $d>2$ and $A$ be the $n\times n$ matrix $(a_{ij})$. By the inductive step, we may assume that $\frac{g}{x_i}$ is a square-free monomial for some $i$. By Lemma \ref{4}, we have $\det A[j:~j\in A_g]\neq0$. So there exist scalars $(c_j)_{j\in A_g}$, such that $\sum_{j\in A_g}c_jl_j\equiv x_i\pmod{\langle x_j:~j\notin A_g\rangle}$. It follows that $x_i=\sum_{j\in A_g}c_jl_j+\sum_{j\notin A_g}c_jx_j$, where $c_j\in k$ for all $j\notin A_g$. Then $g=\sum_{j\in A_g}c_jl_j\frac{g}{x_i}+\sum_{j\notin A_g}c_jx_j\frac{g}{x_i}$. Let $h=\sum_{j\notin A_g}c_jx_j\frac{g}{x_i}$. Note that $h\neq0$ is a $k$-linear combination of square-free monomials of degree $d$. Since $\sum_{j\in A_g}c_jl_j\frac{g}{x_i}\in P$, we obtain that $g\equiv h\pmod P$.
\end{proof}

By the proof of Lemma \ref{5}, we obtain the following.

\begin{rem}\label{12}
Let $P$ be as in Lemma \ref{5} and $0\leq d\leq n$. If $w$ is a square-free monomial in $S_d$ and $q\in S_1$, then $$qw=\widetilde{q}w+\widehat{q}w$$ where $\widetilde{q},\widehat{q}\in S_1$, $\widehat{q}w\in P$ and $\widetilde{q}w$ is a $k$-linear combination of square-free monomials.
\end{rem}

\begin{exa}
Assume that $S=\mathbb{C}[x_1,x_2,x_3]$ and
\begin{align*}
f_1&=x_1^2+x_1x_2+x_1x_3=x_1(x_1+x_2+x_3)\\
f_2&=-x_1x_2+x_2^2+x_2x_3=x_2(-x_1+x_2+x_3)\\
f_3&=-x_1x_3-x_2x_3+x_3^2=x_3(-x_1-x_2+x_3).
\end{align*}

In this case,
$A=\left(\begin{array}{ccc}
1 & 1 & 1\\
-1 & 1 & 1\\
-1 & -1 & 1\end{array}\right)$ is the matrix that defined in Lemma \ref{4}. Since $\det A[i_1,\dots,i_r]\neq0$ for all $1\leq r\leq 3$ and $1\leq i_1<\cdots <i_r\leq 3$, we have that $f_1,f_2,f_3$ is a regular sequence in $S$. Set $P=\langle f_1,f_2,f_3\rangle$ and let $g=x_1^3+x_1^2x_2$. Since $x_1^2\equiv -x_1x_2-x_1x_3 \pmod P$, we have $x_1^3\equiv -x_1^2x_2-x_1^2x_3 \pmod P$. So $g\equiv -x_1^2x_3 \pmod P$. Also, we see that $x_3f_1-x_1f_3=2x_1^{2}x_3+2x_1x_2x_3\in P$. So $g\equiv x_1x_2x_3 \pmod P$ and $x_1x_2x_3\notin \langle x_1^2,x_2^2,x_3^2\rangle$.
\end{exa}

\begin{rem}
Lemma \ref{5} is not true if $f_1,\dots,f_n$ is an arbitrary regular sequence. For example, consider the sequence $$f_1=x_1^2+x_1x_2+x_2^2,~f_2=x_1x_2~ \mathrm{in}~ \mathbb{C}[x_1,x_2].$$ Note that $f_1,f_2$ is a regular sequence $\Leftrightarrow$ $f_1,x_1$ and $f_1,x_2$ are regular sequences $\Leftrightarrow$ $x_1,f_1$ and $x_2,f_1$ are regular sequences $\Leftrightarrow$ $x_2^2$ and $x_1^2$ are regular elements in $\mathbb{C}[x_2]$ and $\mathbb{C}[x_1]$, respectively. So $f_1,f_2$ is a regular sequence. Let $g=x_2^2$. It is easy to show that $g\notin \langle f_1,f_2\rangle$. If $g\equiv ax_1x_2 \pmod{\langle f_1,f_2\rangle}$, for some $a\in \mathbb{C}$, then there exist $c_1,c_2,c_3\in \mathbb{C}$, not all zero, such that $c_1f_1+c_2f_2+c_3(g-ax_1x_2)=0$. But the equation $$c_1x_1^2+(c_1+c_2-ac_3)x_1x_2+(c_1+c_3)x_2^2=0,$$ implies that $c_1=c_2=c_3=0$, a contradiction.

\end{rem}

As a result of Lemma \ref{5}, we obtain the following.

\begin{lem}\label{8}
If $P$ as in Lemma \ref{5}, then the set of all square-free monomials form a $k$-basis of $S/P$.
\end{lem}

\begin{proof}
Denote by $\mathcal{A}$ the set of all square-free monomials in $S$. Lemma \ref{5} shows that $S/P$ generated by $\mathcal{A}$. Let $w=x_1\cdots x_n$. Assume that $w\in P$. Since $H(S/P)=H(S/\langle x_1^2,\dots,x_n^2\rangle)$, it follows that there is a polynomial $f\in S_n$ such that $f\notin P$. By Lemma \ref{5}, $f\equiv bx_1\cdots x_n \pmod P$, where $0\neq b\in k$. Since $w\in P$, it follows that $f\in P$, a contradiction. So $w\notin P$. Suppose that $\sum_{w\in \mathcal{A}}a_ww\in P$, where $a_w\in k$ and $a_w=0$ for almost all $w\in \mathcal{A}$. Assume that $a_w\neq0$ for some $w$. Let $v\in\mathcal{A}$ be a monomial with minimal degree such that $a_v\neq0$. So $\overline{v}\in \langle \overline{f_i}:~i\in A_{v}\rangle$ in the ring $S/\langle x_i:~i\notin A_{v}\rangle$, a contradiction.
\end{proof}

\begin{lem}\label{13}
Let $P$ be as in Lemma \ref{5}. If $w$ is a square-free monomial in $S_d$, where $0\leq d\leq n$, then
\begin{itemize}
  \item[(a)] $|S_1(w)\cap P_{d+1}|=d$.
  \item[(b)] $|S_1(w)\cap (P_{d+1}+S_1(w_1,\dots,w_t))|=|S_1(w)\cap P_{d+1}|+|S_1(w)\cap S_1(w_1,\dots,w_t)|$ for every square-free monomials $w_1,\dots,w_t$ of degrees $d$ such that $w_i\neq w$ for all $1\leq i\leq t$.
\end{itemize}
\end{lem}

\begin{proof}
(a). Let $q=\sum_{i=1}^{n}c_il_i\in S_1$, where $c_i\in k$ for all $i$, such that $qw\in P_{d+1}$. Assume that $c_j\neq0$ for some $j\notin A_w$. Since $qw\prod_{j\neq k\notin A_w}x_k\in P$, it follows that $c_jl_jw\prod_{j\neq k\notin A_w}x_k\in P$. Thus, $c_jl_jw\prod_{j\neq k\notin A_w}x_k=h_1f_1+\dots+h_nf_n$, where $h_i\in S$ for all $1\leq i\leq n$. So $$h_1f_1+\dots+h_{j-1}f_{j-1}+(x_jh_j-c_jw\prod_{j\neq k\notin A_w}x_k)l_j+h_{j+1}f_{j+1}+\dots+h_nf_n=0,$$ which implies that $$x_jh_j-c_jw\prod_{j\neq k\notin A_w}x_k\in \langle f_1,\dots,f_{j-1},f_{j+1},\dots,f_n\rangle.$$ So $\overline{w\prod_{j\neq k\notin A_w}x_k}\in \langle \overline{f_1},\dots,\overline{f_{j-1}},\overline{f_{j+1}},\dots,\overline{f_n}\rangle$ in the ring $S/\langle x_j\rangle$, a contradiction to Lemma \ref{8}. It follows that $q$ belong to the $k$-vector space $(l_i:~i\in A_w)$. On the other hand, $l_iw\in P$, for all $i\in A_w$. So $|S_1(w)\cap P_{d+1}|=\dim (l_iw:~i\in A_w)=d$.

(b). First, we show that $$S_1(w)\cap (P_{d+1}+S_1(w_1,\dots,w_t))=S_1(w)\cap P_{d+1}+S_1(w)\cap S_1(w_1,\dots,w_t).$$ Assume that $qw\in P_{d+1}+S_1(w_1,\dots,w_t)$, where $q\in S_1$. There exist $f\in S_1(w_1,\dots,w_t)$ and $g\in P_{d+1}$ such that $qw=g+f$. If $f\in P$, then $qw\in S_1(w)\cap P_{d+1}$. So assume that $f\notin P$. By \ref{12}, we may assume that $f$ is a $k$-linear combination of square-free monomials. Also, we obtain that $qw=\widetilde{q}w+\widehat{q}w$, where $\widetilde{q},\widehat{q}\in S_1$, $\widehat{q}w\in P$ and $\widetilde{q}w$ is a $k$-linear combination of square-free monomials. So $\widetilde{q}w-f\in P$, which implies that $\widetilde{q}w=f\in S_1(w_1,\dots,w_t)$. Hence $qw\in S_1(w)\cap P_{d+1}+S_1(w)\cap S_1(w_1,\dots,w_t)$ and we obtain that the desired equality.\\ It remains to show that $$S_1(w)\cap S_1(w_1,\dots,w_t)\cap P_{d+1}=(0).$$ Let $qw\in S_1(w_1,\dots,w_t)\cap P_{d+1}$, where $q\in S_1$. By (a), we have $q=\sum_{j\in A_w}c_jl_j$, where $c_j\in k$ for all $j\in A_w$. For every $1\leq j\leq t$, let $i_j\in A_{w_j}\setminus A_w$ and let $B=\{i_j:~1\leq j\leq t\}$. By the hypothesis, we obtain that $qw=\sum_{i=1}^tq_iw_i$, where $q_i\in S_1$ for all $1\leq i\leq t$. So $\overline{qw}=\overline{0}$ in the ring $S/\langle x_j:~j\in B\rangle$, which implies that $\sum_{j\in A_w}\overline{c_jl_j}=\overline{0}$. By \ref{4}, we obtain that $c_j=0$, for all $j\in A_w$. Thus, $qw=0$.

\end{proof}

\begin{rem}
Part (b) of Lemma \ref{13} is not true if we replace $w,w_1,\dots,w_t$ by homogeneous polynomials which are a $k$-linear combination of square-free monomials in $S_d$. For example, let $S=k[x_1,x_2,x_3,x_4]$ and $P=\langle x_1^2,x_2^2,x_3^2,x_4^2\rangle$. Suppose that $h=x_1x_2+x_2x_4+x_3x_4$ and $h_1=x_1x_2+x_1x_3$. Computation with Macaulay2 shows that \begin{center}$|S_1(h)\cap(P_{3}+S_1(h_1))|=2$ and $|S_1(h)\cap P_{3}|=|S_1(h)\cap S_1(h_1))|=0$.\end{center}
\end{rem}

In the case that $w$ is a homogeneous polynomial in part (a) of Lemma \ref{13}, the dimension is always bounded by the degree. This is a result of the following proposition.

\begin{prop}
Let $P$ be as in Lemma \ref{5}. If $g\notin P$ is a homogeneous polynomial of degree $d$, then $|S_1(g)\cap P_{d+1}|\leq d$.
\end{prop}

\begin{proof}
We prove by induction on $n$. If $n=1$, then $g=ax_1$ or $g\in k$, where $a\in k$. If $g\in k$, then $|S_1(g)\cap P_{1}|=0$ and if $g=ax_1$, then $|S_1(g)\cap P_{2}|=1$. Let $n>1$. We prove by induction on $d$, starting with $d=0$. Let $d>0$. If $d=n$, then $P_{d+1}=S_{d+1}$ and so $|S_1(g)\cap P_{d+1}|=n$. Assume that $d<n$. By \ref{5}, there exists a $k$-linear combination of square-free monomials $h\in S_d$ such that $g\equiv h\pmod{P_{d}}$. Clearly, $S_1(h)\cap P_{d+1}=S_1(g)\cap P_{d+1}$. Let $h=\sum_{i=1}^{t}a_iw_i$, where $0\neq a_i\in k$ and $w_i\in \Mon(S_d)$ for all $i$. Let $j\notin A_{w_1}$. If $l_jh\in P_{d+1}$, then $\overline{l_jw_1}\in \overline{P_{d+1}}$ in the ring $S/\langle x_i:~i\notin A_w \wedge i\neq j\rangle$, a contradiction. So $l_jh\notin P_{d+1}$ for all $j\notin A_{w_1}$. In particular, there exists a variable $x_i$ such that $x_ih\notin P_{d+1}$. We have two cases:

\textbf{Case 1.} $\overline{h}\notin \overline{P_{d}}$ in the ring $S/\langle x_i\rangle$. Let $\overline{p_1h},\dots,\overline{p_sh}$ be a basis of $\overline{S_1}(\overline{h})\cap \overline{P_{d+1}}$ in the ring $S/\langle x_i\rangle$. By the inductive step, we obtain that $s\leq d$. If $f\in S_1(h)\cap P_{d+1}$, then $f\in (p_1h,\dots,p_sh,x_iq)$, where $q\in S_d$. Since $f\in S_1(h)$, it follows that $x_iq=rh$, where $r\in S_1$. Since $x_i\nmid h$, it follows that $x_i|r$. So $f\in (p_1h,\dots,p_sh,x_ih)$. Therefore, $S_1(h)\cap P_{d+1}\subseteq (p_1h,\dots,p_sh,x_ih)$. If $|S_1(h)\cap P_{d+1}|=s+1$, then $x_ih\in P_{d+1}$, a contradiction.

\textbf{Case 2.} $\overline{h}\in \overline{P_{d}}$ in the ring $S/\langle x_i\rangle$. So $h\equiv x_iq \pmod{P_d}$, where $q\in S_{d-1}$. Since $h$ is the unique $k$-linear combination of square-free monomials such that $x_iq\equiv h \pmod P$, we obtain that $h=x_ih_1$, where $h_1\in S_{d-1}$. If $f\in S_1(h)\cap P_{d+1}$, then $f=px_ih_1$, for some $p\in S_1$. Clearly, $\frac{f}{x_i}\in S_1(h_1)$. Since $f\in P$, it follows that $\overline{ph_1}\in \overline{P_{d}}$ in the ring $S/\langle l_i\rangle$. So $\overline{\frac{f}{x_i}}\in \overline{S_1}(\overline{h_1})\cap \overline{P_{d}}$ in $S/\langle l_i\rangle$. If $\overline{h_1}\in\overline{P_{d-1}}$ in $S/\langle l_i\rangle$, then $x_ih_1=h\in P$, a contradiction. Let $\overline{p_1h_1},\dots,\overline{p_sh_1}$ be a basis of $\overline{S_1}(\overline{h_1})\cap \overline{P_{d}}$. By the inductive step, we obtain that $s\leq d-1$. So $\frac{f}{x_i}\in (p_1h_1,\dots,p_sh_1,l_iq)$, which implies that $f\in (p_1h,\dots,p_sh,l_ix_iq)$. Therefore, $|S_1(h)\cap P_{d+1}|\leq s+1\leq d$.
\end{proof}
Now, we prove the main results of this section.

\begin{thm}
Let $P$ be as in Lemma \ref{5} and $M=\langle x_1^2,\dots,x_n^2\rangle$. Assume that $V=P_{d}+(w_1,\dots,w_t)$ and $W=M_{d}+(w_1,\dots,w_t)$, where $w_i$ is a square-free monomial of degree $d$, for all $i$. Then $$\dim S_1W=\dim S_1V.$$
\end{thm}

\begin{proof}
We may assume that $d\geq2$ and prove by induction on $t$. If $t=1$, then
\begin{align*}
\dim S_1W&=\dim M_{d+1}+\dim S_1(w_1)-\dim S_1(w_1)\cap M_{d+1}\\
&=\dim P_{d+1}+\dim S_1(w_1)-\dim S_1(w_1)\cap P_{d+1}=\dim S_1V.
\end{align*}
Let $t>1$, and set $W_1=M_{d}+(w_1,\dots,w_{t-1})$, $V_1=P_{d}+(w_1,\dots,w_{t-1})$ and $Z=S_1(w_t)\cap S_1(w_1,\dots,w_{t-1})$. By Lemma (\ref{13}) and the inductive step, we have
\begin{align*}
\dim S_1W&= \dim S_1W_1+\dim S_1(w_t)-\dim S_1(w_t)\cap S_1W_1\\
&=\dim S_1V_1+\dim S_1(w_t)-\dim S_1(w_t)\cap S_1W_1\\
&=\dim S_1V_1+\dim S_1(w_t)-\dim S_1(w_t)\cap M_{d+1}-\dim Z\\
&=\dim S_1V_1+\dim S_1(w_t)-\dim S_1(w_t)\cap P_{d+1}-\dim Z\\
&=\dim S_1V_1+\dim S_1(w_t)-\dim S_1(w_t)\cap S_1V_1\\
&=\dim S_1V.
\end{align*}
\end{proof}

\begin{prop}
Let $P$ be as in Lemma \ref{5} and $V=P_{d}+(w_1,\dots,w_t)$ be the $k$-vector space spanned by $P_{d}$ and the $t$ biggest (in lex order) square-free monomials in $S_d$. Then $$\dim S_1V={d+n \choose d+1}-{n\choose d+1}+\sum_{i=1}^{t}(n-m(w_i)),$$ where $m(w_i)=\max\{j:~x_j|w_i\}$, $1\leq i\leq t$.
\end{prop}

\begin{proof}
We claim that
$$|S_1V|=|P_{d+1}|+\sum_{i=1}^t|S_1(w_i)|-\sum_{i=1}^t|S_1(w_i)\cap P_{d+1}|-\sum_{i=2}^{t}|S_1(w_i)\cap S_1(w_1,\dots,w_{i-1})|.$$
We prove the claim by induction on $t$. If $t=1$, then $$|S_1V|=|P_{d+1}|+|S_1(w_1)|-|S_1(w_1)\cap P_{d+1}|.$$ Let $t>1$ and $V_1=P_{d}+(w_1,\dots,w_{t-1})$. By the inductive step we obtain that $|S_1V|$ is equal to
$$|P_{d+1}|+\sum_{i=1}^{t}|S_1(w_i)|-\sum_{i=1}^{t-1}|S_1(w_i)\cap P_{d+1}|-\sum_{i=2}^{t-1}|S_1(w_i)\cap S_1(w_1,\dots,w_{i-1})|-|S_1(w_t)\cap S_1\overline{V}|.$$
By Lemma \ref{13}, we have $|S_1(w_t)\cap S_1V_1|=|S_1(w_t)\cap P_{d+1}|+|S_1(w_t)\cap S_1(w_1,\dots,w_{t-1})|$. We proved the claim.

Let $2\leq j\leq t$. If $i<m(w_j)$ such that $x_i\nmid w_j$, then $x_iw_j\in S_1(w_1,\dots,w_{j-1})$. So $|S_1(w_i)\cap S_1(w_1,\dots,w_{i-1})|=m(w_j)-d$. Therefore
\begin{align*}
|S_1V|&=|S_{d+1}|-{n\choose d+1}+tn-td-\sum_{i=2}^{t}(m(w_i)-d)\\
&={d+n \choose d+1}-{n\choose d+1}+tn-td-\sum_{i=2}^{t}(m(w_i)-d)\\
&={d+n \choose d+1}-{n\choose d+1}+tn-td-\sum_{i=1}^{t}(m(w_i)-d)\\
&={d+n \choose d+1}-{n\choose d+1}+\sum_{i=1}^{t}(n-m(w_i)).
\end{align*}

\end{proof}

\section{The main result}
In this section we prove that the EGH Conjecture is true if $f_i$ splits into linear factors for all $i$. We begin with the following lemma.

\begin{lem}\label{20}
Let $P=\langle f_1,\dots,f_n\rangle$ be an ideal of $S$ generated by a regular sequence with $\deg(f_i)=a_i$ and $n\geq2$. Assume that $f_n=q_1\cdots q_s$, where $q_1,\dots,q_s\in S_1$. Then
\begin{itemize}
  \item[(a)] $H(S/P+\langle q_m\rangle)=H(S/P+\langle q_k\rangle)$ for all $1\leq m,k\leq s$.
  \item[(b)] $H(S/(P:q_1\cdots q_j)+\langle q_m\rangle)=H(S/(P:q_1\cdots q_j)+\langle q_k\rangle)$ for all $1\leq j\leq s-1$ and $j<m,k\leq s$.
\end{itemize}
\end{lem}

\begin{proof}
First, we will prove (a). Let $1\leq m,k\leq s$. Note that $P+\langle q_m\rangle/\langle q_m\rangle$ and $P+\langle q_k\rangle/\langle q_k\rangle$ are ideals in $S/\langle p_m\rangle$ and $S/\langle q_k\rangle$, respectively, generated by $\overline{f_1},\dots,\overline{f_{n-1}}$. Note also that $f_1,\dots,f_{n-1},q_m$ and $f_1,\dots,f_{n-1},q_k$ are regular sequences. By part (c) of Lemma \ref{3}, we obtain that $\overline{f_1},\dots,\overline{f_{n-1}}$ is a regular sequence in $S/\langle q_m\rangle$ and $S/\langle q_k\rangle$. By part (a) of Lemma \ref{3}, we obtain that $H(S/P+\langle q_m\rangle)=H(S/P+\langle q_k\rangle)$.

Now, we prove (b). Let $1\leq j\leq s-1$ and $j<m,k\leq s$. Assume that $$h=h_1+h_2\in (P:q_1\cdots q_j)+\langle q_m\rangle,$$ where $h_1\in (P:q_1\cdots q_j)$ and $h_2\in \langle q_m\rangle$. Since $h_1q_1\cdots q_j\in P$, it follows that $h_1q_1\cdots q_j=g_1f_1+\dots+g_nf_n$, where $g_1,\dots,g_n\in S$; i.e., $$g_1f_1+\dots+g_{n-1}f_{n-1}+q_1\cdots q_j(g_nq_{j+1}\cdots q_s-h_1)=0.$$ Since $f_1,\dots,f_{n-1},q_1\cdots q_j$ is a regular sequence, it follows that $g_nq_{j+1}\cdots q_s-h_1\in \langle f_1,\dots,f_{n-1}\rangle$. So $\overline{h_1}\in \langle \overline{f_1},\dots,\overline{f_{n-1}}\rangle$ in the ring $S/\langle q_m\rangle$, which implies that $\overline{h}\in \langle \overline{f_1},\dots,\overline{f_{n-1}}\rangle$. Conversely, $\overline{f_i}\in (P:q_1\cdots q_j)+\langle q_m\rangle/\langle q_m\rangle$ for all $1\leq i\leq n-1$. So $(P:q_1\cdots q_j)+\langle q_m\rangle/\langle q_m\rangle$ is an ideal in $S/\langle q_m\rangle$ generated by $\overline{f_1},\dots,\overline{f_{n-1}}$. Similarly, $(P:q_1\cdots q_j)+\langle q_k\rangle/\langle q_k\rangle$ is an ideal in $S/\langle q_k\rangle$ generated by $\overline{f_1},\dots,\overline{f_{n-1}}$. By Lemma \ref{3}, it follows that $H(S/(P:q_1\cdots q_j)+\langle q_k\rangle)=H(S/(P:q_1\cdots q_j)+\langle q_m\rangle)$.
\end{proof}

\begin{thm}\label{21}
Assume that the EGH Conjecture holds in $k[x_1,\dots,x_{n-1}]$, where $n\geq2$. If $I$ is a graded ideal in $S=k[x_1,\dots,x_{n}]$ containing a regular sequence $f_1,\dots,f_{n-1},f_n=q_1\cdots q_s$ of degrees $\deg(f_i)=a_i$ such that $q_i\in S_1$ for all $1\leq i\leq s$, then $I$ has the same Hilbert function as a graded ideal in $S$ containing $x_1^{a_1},\dots,x_n^{a_n}$.
\end{thm}

\begin{proof}
We check the property (b) of Lemma \ref{2}. Let $d\geq0$. We need to find a graded ideal $K$ in $S$ containing $x_1^{a_1},\dots,x_n^{a_n}$ such that $H(S/I,d)=H(S/K,d)$ and $H(S/I,d+1)\leq H(S/K,d+1)$. Let $J$ to be the ideal generated by $f_1,\dots,f_n$ and $I_d$. By renaming the linear polynomials $q_1,\dots,q_s$, we may assume without loss of generality
that
\begin{center} $|J_d\cap \langle q_1\rangle_d|\geq |J_d\cap \langle q_i\rangle_d|$ for all $2\leq i\leq s$,\\
$|(J:q_1)_{d-1}\cap \langle q_2\rangle_{d-1}|\geq |(J:q_1)_{d-1}\cap \langle q_i\rangle_{d-1}|$ for all $3\leq i\leq s$,\\
$|(J:q_1q_2)_{d-2}\cap \langle q_3\rangle_{d-2}|\geq |(J:q_1q_2)_{d-2}\cap \langle q_i\rangle_{d-2}|$ for all $4\leq i\leq s$,\\
$\vdots$\\
$|(J:q_1\cdots q_{s-2})_{d-(s-2)}\cap \langle q_{s-1}\rangle_{d-(s-2)}|\geq |(J:q_1\cdots q_{s-2})_{d-(s-2)}\cap \langle q_s\rangle_{d-(s-2)}|$.
\end{center}
By considering the short exact sequences
\begin{center}
$0\rightarrow S/(J:q_1)\underset{g\mapsto gq_1}{\longrightarrow}S/J\underset{g\mapsto g}{\longrightarrow}S/J+\langle q_1\rangle\rightarrow 0$,\\
$0\rightarrow S/(J:q_1q_2)\underset{g\mapsto gq_2}{\longrightarrow}S/(J:q_1)\underset{g\mapsto g}{\longrightarrow}S/(J:q_1)+\langle q_2\rangle\rightarrow 0$,\\
$0\rightarrow S/(J:q_1q_2q_3)\underset{g\mapsto gq_3}{\longrightarrow}S/(J:q_1q_2)\underset{g\mapsto g}{\longrightarrow}S/(J:q_1q_2)+\langle q_3\rangle\rightarrow 0$,\\
$\vdots$\\
$0\rightarrow S/(J:q_1\cdots q_{s-1})\underset{g\mapsto gq_{s-1}}{\longrightarrow}S/(J:q_1\cdots q_{s-2})\underset{g\mapsto g}{\longrightarrow}S/(J:q_1\cdots q_{s-2})+\langle q_{s-1}\rangle\rightarrow 0$.
\end{center}
we see that $H(S/J,t)$ is equal to $$H(S/J+\langle q_1\rangle,t)+\sum_{i=1}^{s-2}H(S/(J:q_1\cdots q_i)+\langle q_{i+1}\rangle,t-i)+H(S/(J:q_1\cdots q_{s-1}),t-(s-1))$$ for all $t\geq0$. Let $J_0=J+\langle q_1\rangle$, $J_{s-1}=(J:q_1\cdots q_{s-1})$, and for $1\leq i\leq s-2$ let $J_i=(J:q_1\cdots q_i)+\langle q_{i+1}\rangle$. Note that $q_{i+1}\in J_i$ and $H(\frac{S/\langle q_{i+1}\rangle}{J_i/\langle q_{i+1}\rangle})=H(S/J_i)$ for all $0\leq i\leq s-1$. Set $\overline{S}=k[x_1,\dots,x_{n-1}]$. For all $0\leq i\leq s-1$, $S/\langle q_{i+1}\rangle$ is isomorphic to $\overline{S}$, so by the hypothesis there is an ideal in $\overline{S}$ containing $x_1^{a_1},\dots,x_{n-1}^{a_{n-1}}$ with the same Hilbert function as $J_{i}$. For all $0\leq i\leq s-1$, let $L_i$ be the lex-plus-powers ideal in $\overline{S}$ containing $x_1^{a_1},\dots,x_{n-1}^{a_{n-1}}$ such that $H(\overline{S}/L_i)=H(S/J_i)$.\\\\
\textbf{\emph{Claim:}} $L_{i,j}\subseteq L_{i+1,j}$ for all $0\leq i\leq s-2$ and $j\leq d-i$, where $L_{i,j}$ is the $j$-th component of the ideal $L_i$.\\\\
\emph{Proof of the claim:} Assume that $i=0$. If $j<d$, then by part (a) of Lemma \ref{20} we obtain
$$|J_{0,j}|=|J_j+\langle q_1\rangle_j|=|P_j+\langle q_1\rangle_j|=|P_j+\langle q_2\rangle_j|\leq |J_{1,j}|.$$ If $j=d$, then by our assumption we obtain
\begin{align*}
|J_{0,d}|&=|J_d|+|\langle q_1\rangle_d|-|J_d\cap \langle q_1\rangle_d|\\
&\leq|J_d|+|\langle q_1\rangle_d|-|J_d\cap \langle q_2\rangle_d|\\
&=|J_d|+|\langle q_2\rangle_d|-|J_d\cap \langle q_2\rangle_d|\\
&=|J_d+\langle q_2\rangle_d|\\
&\leq |J_{1,d}|.
\end{align*}
This means that $H(S/J_0,j)\geq H(S/J_1,j)$ for all $j\leq d$. So $H(\overline{S}/L_0,j)\geq H(\overline{S}/L_1,j)$ for all $j\leq d$. Since $L_0$ and $L_1$ are lex-plus-powers ideals, it follows that $L_{0,j}\subseteq L_{1,j}$ for all $j\leq d$.

Let $0<i\leq s-2$. If $j<d-i$, then by part (b) of Lemma \ref{20} we obtain $$|J_{i,j}|=|(J:q_1\cdots q_i)_j+\langle q_{i+1}\rangle_j|=|(P:q_1\cdots q_i)_j+\langle q_{i+1}\rangle_j|=|(P:q_1\cdots q_i)_j+\langle q_{i+2}\rangle_j|\leq|J_{i+1,j}|.$$ If $j=d-i$, then by our assumption we obtain
\begin{align*}
|J_{i,d-i}|&=|(J:q_1\cdots q_i)_{d-i}|+|\langle q_{i+1}\rangle_{d-i}|-|(J:q_1\cdots q_i)_{d-i}\cap \langle q_{i+1}\rangle_{d-i}|\\
&\leq|(J:q_1\cdots q_i)_{d-i}|+|\langle q_{i+1}\rangle_{d-i}|-|(J:q_1\cdots q_i)_{d-i}\cap \langle q_{i+2}\rangle_{d-i}|\\
&=|(J:q_1\cdots q_i)_{d-i}|+|\langle q_{i+2}\rangle_{d-i}|-|(J:q_1\cdots q_i)_{d-i}\cap \langle q_{i+2}\rangle_{d-i}|\\
&=|(J:q_1\cdots q_i)_{d-i}+\langle q_{i+2}\rangle_{d-i}|\\
&\leq |J_{i+1,d-i}|.
\end{align*}
Similarly, we conclude that $L_{i,j}\subseteq L_{i+1,j}$ for all $j\leq d-i$, and proving the claim.

Let $K_{s}=\{zx_n^{s+j}:~z\in\Mon(\overline{S})\wedge j\geq0\}$ and $K_{i}=\{zx_n^{i}:~z\in \Mon(L_{i})\}$ for all $0\leq i\leq s-1$. Define $K$ to be the ideal generated by $\bigcup_{0\leq i\leq s}K_i$. Since $x_n^{s}\in K_s$ and $x_i^{a_i}\in K_0$ for all $1\leq i\leq n-1$, it follows that $x_1^{a_1},\dots,x_n^{a_n}\in K$.\\\\
\textbf{\emph{Claim:}} If $w$ is a monomial in $K$ of degree $t$, where $0\leq t\leq d+1$, then $w\in \bigcup_{0\leq i\leq s}K_i$.\\\\
\emph{Proof of the claim:} There exists a monomial $u$ in $\bigcup_{0\leq i\leq s}K_i$ such that $u|w$; i.e., $w=vu$ for some monomial $v\in S$. If $u\in K_s$, then $w\in K_s$. Assume that $u=zx_n^{i}\in K_i$, where $z\in L_i$ for some $0\leq i\leq s-1$. If $x_n\nmid v$, then $w\in \bigcup_{0\leq i\leq s}K_i$. Assume that $x_n|v$. Let $r=\max\{j:~x_n^j|v\}$. If $i+r\geq s$, then $w\in K_s$. So we may assume that $i+r<s$. By the previous claim, we obtain that $L_{i,j}\subseteq L_{i+r,j}$ for all $j\leq d-(i+r-1)$. Since $\deg(z)\leq d+1-(i+r)$, it follows that $z\in L_{i+r}$. So $\frac{v}{x_n^r}z\in L_{i+r}$, and then $\frac{v}{x_n^r}zx_n^{r+i}=w\in K_{i+r}$. Hence, we proved the claim.

We conclude that the number of monomials in $K$ of degree $t$, where $0\leq t\leq d+1$, is equal to $\sum_{i=0}^{s-1}|L_{i,t-i}|+\sum_{i=0}^{t-s}|\overline{S}_i|$. Since $|S_t|=\sum_{0\leq i\leq t}|\overline{S}_i|$, it follows that $$|S_t|-|K_t|=\sum_{i=t-(s-1)}^{t}|\overline{S}_{i}|-\sum_{i=0}^{s-1}|L_{i,t-i}|=\sum_{i=0}^{s-1}|\overline{S}_{t-i}|-\sum_{i=0}^{s-1}|L_{i,t-i}|.$$ So $H(S/K,t)=\sum_{i=0}^{s-1}H(\overline{S}/L_i,t-i)=\sum_{i=0}^{s-1}H(S/J_i,t-i)=H(S/J,t).$
In particular,\begin{center} $H(S/K,d)=H(S/J,d)=H(S/I,d)$ and $H(S/K,d+1)=H(S/J,d+1)\geq H(S/I,d+1)$.\end{center}
\end{proof}

\begin{cor}\label{22}
If $I$ is a graded ideal in $S$ containing a regular sequence $f_1,\dots,f_n$ with $\deg(f_i)=a_i$ such that $f_i$ splits into linear factors for all $i$, then $I$ has the same Hilbert function as a graded ideal in $S$ containing $x_1^{a_1},\dots,x_n^{a_n}$.
\end{cor}

Since the EGH Conjecture holds when $n=2$, we obtain the following.

\begin{cor}
Let $n\geq3$. If $I$ is a graded ideal in $S$ containing a regular sequence $f_1,\dots,f_n$ with $\deg(f_i)=a_i$ such that $f_i$ splits into linear factors for all $3\leq i\leq n$, then $I$ has the same Hilbert function as a graded ideal in $S$ containing $x_1^{a_1},\dots,x_n^{a_n}$.
\end{cor}

By \ref{22}, the EGH Conjecture is equivalent to the following conjecture.

\begin{con}
If $I$ is a homogeneous ideal in $S$ containing a regular sequence $f_1,\dots,f_n$ of degrees $\deg(f_i)=a_i$, then $I$ has the same Hilbert function as an ideal containing a regular sequence $g_1,\dots,g_n$ of degrees $\deg(g_i)=a_i$, where $g_i$ splits into linear factors for all $i$.
\end{con}

\begin{exa}
Let $S=\mathbb{C}[x_1,\dots,x_5]$, $f_i=x_i(\sum_{j=1}^{i-1}-x_j)+x_i(\sum_{j=i}^{5}x_j)$ for all $1\leq i\leq 5$ and $$A=\left(\begin{array}{ccccc}
1 & 1 & 1 & 1 & 1\\
-1 & 1 & 1 & 1& 1\\
-1 & -1 & 1 & 1 & 1\\
-1 & -1 & -1 & 1 & 1\\
-1 & -1 & -1 & -1 & 1
\end{array}\right).$$ Since $\det A[i_1,\dots,i_r]\neq0$ for all $1\leq r\leq 5$ and $1\leq i_1<\cdots <i_r\leq 5$, it follows that $f_1,\dots,f_5$ is a regular sequence in $S$. Assume that $I=\langle f_1,\dots,f_5,x_1x_2+x_1x_3,x_1^2+x_4x_5\rangle$. In this example, we construct an ideal in $S$ with the same Hilbert function as $I$, using the Hilbert functions of $J_0=I+\langle x_5\rangle$ and $J_1=(I:x_5)$. Computation with Macaulay2 shows that $$H_{S/I}=(1,5,8,3,0,0,\dots), H_{S/J_0}=(1,4,4,1,0,0,\dots)~\mathrm{and}~H_{S/J_1}=(1,4,2,0,0,\dots)$$ are the Hilbert sequence of $I$, $J_0$ and $J_1$, respectively. Denote by $R$ the polynomial ring $\mathbb{C}[x_1,\dots,x_4]$. Let \begin{center}$L_0=\langle x_1^2,\dots,x_4^2,x_1x_2,x_1x_3\rangle\subset R$ and $L_1=\langle x_1^2,\dots,x_4^2,x_1x_2,x_1x_3,x_1x_4,x_2x_3\rangle\subset R$.\end{center} Note that $L_0$ and $L_1$ are lex-plus-powers ideals in $R$. We can see that $L_{0,0}=L_{0,1}=(0)$ and \begin{center}$L_{0,2}=(x_1^2,x_1x_2,x_1x_3,x_2^2,x_3^2,x_4^2)$,\\$L_{0,3}=(w:~w\in \Mon(R_3)~\mathrm{and}~w\neq x_2x_3x_4)$,\\
$L_{0,j}=R_j$ for all $j\geq4$.
\end{center}
So we have $H_{R/L_0}=H_{S/J_0}$. Also, we have $L_{1,0}=L_{1,1}=(0)$ and
\begin{center}$L_{1,2}=(x_1^2,x_1x_2,x_1x_3,x_1x_4,x_2^2,x_2x_3,x_3^2,x_4^2)$,\\
$L_{1,j}=R_j$ for all $j\geq3$.
\end{center}
So we have $H_{R/L_1}=H_{S/J_1}$. Let $K$ to be the ideal in $S$ generated by $$\Mon(L_0)\cup \{wx_5:~w\in\Mon(L_1)\}\cup\{x_5^2\}.$$ Then $K=\langle x_1^2,x_2^2,x_3^2,x_4^2,x_5^2,x_1x_2,x_1x_3,x_1x_4x_5,x_2x_3x_5\rangle$. It is clear that $|S_0/K_0|=1$ and $|S_1/K_1|=5$. Since $S_2/K_2=(\overline{x_1x_4},\overline{x_1x_5},\overline{x_2x_3},\overline{x_2x_4},\overline{x_2x_5},\overline{x_3x_4},\overline{x_3x_5},\overline{x_4x_5})$, it follows that $|S_3/K_3|=8$. Also we have $S_3/K_3=(\overline{x_2x_3x_4},\overline{x_2x_4x_5},\overline{x_3x_4x_5})$ and $K_j=S_j$ for all $j\geq4$. Thus $$H_{S/K}=(1,5,8,3,0,0,\dots)=H_{S/I}.$$

\end{exa}

\begin{exa}
Let $S=\mathbb{C}[x_1,\dots,x_6]$, $f_i=x_i(\sum_{j=1}^{i-1}-x_j)+x_i(\sum_{j=i}^{6}x_j)$ for all $1\leq i\leq 5$ and $f_6=x_6^2(-x_1-x_2-x_3-x_4-x_5+x_6)$. Since $f_1,\dots,f_5,\frac{f_6}{x_6}$ is a regular sequence, it follows that $f_1,\dots,f_6$ is a regular sequence in $S$. Assume that $$I=\langle f_1,\dots,f_6,x_1x_2+x_3x_4,x_1x_6+x_5^2,x_2^2x_3\rangle.$$ Computation with Macaulay2 shows that
\begin{align*}
H_{S/I}=&(1,6,14,13,2,0,\dots),\\
H_{S/I+\langle x_6\rangle}=&(1,5,8,2,0,\dots),\\
H_{S/(I:x_6)+\langle x_6\rangle}=&(1,5,6,0,\dots),\\
H_{S/(I:x_6^2)}=&(1,5,2,0,\dots).
\end{align*}
Also we have $$|I_2\cap \langle x_6\rangle_2|=|I_2\cap \langle -x_1-x_2-x_3-x_4-x_5+x_6\rangle_2|$$ and $$|(I:x_6)_1\cap\langle x_6\rangle_1|=|(I:x_6)_1\cap\langle -x_1-x_2-x_3-x_4-x_5+x_6\rangle_1|.$$
We construct an  ideal in $S$ with the same Hilbert function as $I$, using the Hilbert functions of $I+\langle x_6\rangle$, $(I:x_6)+\langle x_6\rangle$ and $(I:x_6^2)$. Denote by $J_0$, $J_1$ and $J_2$ the ideals $I+\langle x_6\rangle$, $(I:x_6)+\langle x_6\rangle$ and $(I:x_6^2)$, respectively. Let $R=\mathbb{C}[x_1,\dots,x_5]$ and $L_0=\langle x_1^2,\dots,x_5^2,x_1x_2,x_1x_3,x_1x_4x_5,x_2x_3x_4,x_2x_3x_5\rangle\subset R$. An easy calculation shows that $L_0$ is a lex-plus-powers ideal in $R$ and $H_{R/L_0}=(1,5,8,0,\dots)=H_{S/I+\langle x_6\rangle}$. Let $L_1=\langle x_1^2,\dots,x_5^2,x_1x_2,x_1x_3,x_1x_4,x_1x_5,x_2x_3x_4,x_2x_3x_5,x_2x_4x_5,x_3x_4x_5\rangle\subset R$. We can see that $L_1$ is a lex-plus-powers ideal and $H_{R/L_1}=(1,5,6,0,\dots)=H_{S/J_1}$. Let $L_2=\langle x_1^2,\dots,x_5^2,x_1x_2,x_1x_3,x_1x_4,x_1x_5,x_2x_3,x_2x_4,x_2x_5,x_3x_4\rangle\subset R$. Also we have that $L_2$ is a lex-plus-powers ideal in $R$ and $H_{R/L_2}=(1,5,3,0,\dots)=H_{S/J_2}$. Let $K$ to be the ideal in $S$ generated by $\Mon(L_0)\cup\{wx_6:~w\in\Mon(L_1)\}\cup\{wx_6^2:~w\in\Mon(L_2)\}\cup\{x_6^3\}$. The ideal $K$ generated by $$\{x_1^2,\dots,x_5^2,x_6^3,x_1x_2,x_1x_3,x_1x_4x_5,x_2x_3x_4,x_2x_3x_5,x_1x_4x_6\}$$$$\bigcup$$$$\{x_1x_5x_6,x_2x_4x_5x_6,x_3x_4x_5x_6,x_2x_3x_6^2,x_2x_4x_6^2,x_2x_5x_6^2,x_3x_4x_6^2\}.$$ Computation with Macaulay2 shows that $H_{S/K}=(1,6,14,13,2,0,\dots)=H_{S/I}$.
\end{exa}

\end{document}